\documentclass[11pt, twoside, reqno]{amsart}
\usepackage{xcolor, amstext, amsfonts, amssymb, amsbsy, latexsym}
\usepackage{enumerate}
\usepackage[T1]{fontenc}
\usepackage{xy, hhline}
\newcounter{num}[section] %

\newenvironment{theo}
{\refstepcounter{num}%
\bigskip\noindent{\bf Theorem~\arabic{section}.\arabic{num}. }\it}
{\smallskip}

\newenvironment{lemma}
{\refstepcounter{num}%
\bigskip\noindent{\bf Lemma~\arabic{section}.\arabic{num}. }\it}
\newenvironment{eq}{\begin{equation}}{\end{equation}}

\newcommand{\si}{\sigma}
\newcommand{\al}{\alpha}
\newcommand{\be}{\beta}
\newcommand{\ga}{\gamma}

\newcommand{\de}{\delta}

\newcommand{\un}[1]{{\underline{#1}} }

\newcommand{\tr}{\mathop{\rm tr}}

\newcommand{\mdeg}{\mathop{\rm mdeg}}
\newcommand{\Char}{\mathop{\rm char}}

\newcommand{\FF}{{\mathbb{F}}}   % base field
\newcommand{\NN}{{\mathbb{N}}}
\newcommand{\I}{\mathbb{I}} % i^2=-1
\newcommand{\symmmatr}[6]{\left(
\begin{array}{ccc}
#1 & #2 & #3\\ 
#2 & #4 & #5\\ 
#3 & #5 & #6 \\
\end{array}
\right)}

\newcommand{\mylabel}[1]{}

\newcommand{\mycomment}[1]{}
\begin{document}
\title[Minimal generating and separating sets for orthogonal invariants]{Minimal generating and separating sets for $O(3)$-invariants of several matrices}

\thanks{This research was supported by FAPESP 2016/00541-0 and FAEPEX 2449/18}
%\thanks{This research was supported by RFBR 16-31-60111 (mol\_a\_dk).}

\author{Ronaldo Jos\'e Sousa Ferreira}
\address{Ronaldo Jos\'e Sousa Ferreira\\ 
Federal University of Maranh\~ao, Rua Santa Clara, 2010,
65940-000 Graja\'u, MA, Brazil}
\email{ronaldoj.sf@hotmail.com (Ronaldo Jos\'e Sousa Ferreira)}

\author{Artem Lopatin}
\address{Artem Lopatin\\ 
Universidade Estadual de Campinas (UNICAMP), 651 Sergio Buarque de Holanda, 13083-859 Campinas, SP, Brazil}
\email{dr.artem.lopatin@gmail.com (Artem Lopatin)}
%\address{Artem Lopatin\\
%Sobolev Institute of Mathematics, Omsk Branch, SB RAS, Omsk, Russia}
%\email{artem\_lopatin@yahoo.com}

\begin{abstract} 
Given an algebra  $\FF[H]^G$ of polynomial invariants of an action of the group $G$ over the vector space $H$, 
a subset $S$ of $\FF[H]^G$ is called separating if $S$ separates all orbits that can be separated by $\FF[H]^G$. 
A minimal separating set is found for some algebras of matrix invariants of several matrices over an infinite field of arbitrary characteristic different from two in case of the orthogonal group. Namely, we consider the following cases:
\begin{enumerate}
\item[$\bullet$] $GL(3)$-invariants of two matrices;

\item[$\bullet$]  $O(3)$-invariants of $d>0$ skew-symmetric matrices; 

\item[$\bullet$]  $O(4)$-invariants of two skew-symmetric matrices;

\item[$\bullet$]  $O(3)$-invariants of two symmetric matrices.
\end{enumerate}
A minimal generating set is also  given for the algebra of orthogonal invariants of three $3\times 3$ symmetric matrices.


\end{abstract}

\maketitle

%========================================================
%S1======================================================
\section{Introduction}\label{section_intro}

\subsection{Definitions} All vector spaces, algebras, and modules are over an infinite field $\FF$ of an arbitrary characteristic $p=\Char{\FF}\geq0$, unless otherwise stated.  By an algebra we always mean an associative algebra.

Given $n>1$ and $d\geq 1$,  we consider  the polynomial algebras 
$$\begin{array}{lcll}
R & = &\FF[x_{ij}(k) \;| \;1\leq i,j\leq n,& 1\leq k\leq d]; \\
R_{+} & = &\FF[x_{ij}(k)\;|\;1\leq j\leq i\leq n,& 1\leq k\leq d]; \\
R_{-} & = &\FF[x_{ij}(k)\;|\;1\leq j<i\leq n,& 1\leq k\leq d]. \\
\end{array}
$$
together with $n\times n$ {\it generic} matrices $X_k$,  {\it symmetric  generic} matrices $Y_k$ and {\it skew-symmetric generic} matrices  $Z_k$:
$$(X_k)_{ij}=x_{ij}(k),\qquad
(Y_k)_{ij}=
\left\{ 
\begin{array}{rc}
x_{ij}(k),& \text{ if } i\geq j\\
x_{ji}(k),& \text{ if } i< j\\
\end{array}
\right.,\qquad
(Z_k)_{ij}=
\left\{ 
\begin{array}{rc}
x_{ij}(k),& \text{ if } i> j\\
0 ,& \text{ if } i= j\\
-x_{ji}(k),& \text{ if } i< j\\
\end{array}
\right..
$$
Here $(A)_{ij}$ stands for the $(i,j)^{\rm th}$ entry of a matrix $A$.
The $t^{\rm th}$ coefficient of the characteristic polynomial of an $n\times n$ matrix $A$ is denoted by $\si_t(A)$. As an example, $\tr(A)=\si_1(A)$ and $\det(A)=\si_n(A)$.  Denote by $M(n)$ the space of all $n\times n$ matrices over $\FF$, $S_{+}(n)=\{A\in M(n)\,|\,A^T=A\}$, $S_{-}(n)=\{A\in M(n)\,|\,A^T=-A\}$ and $O(n)=\{A\in M(n)\,|\,AA^T=I_n\}$. Consider the algebras of {\it matrix invariants} $R^{GL(n)}$, $R^{O(n)}$, $R_{+}^{O(n)}$, $R_{-}^{O(n)}$, respectively, that are generated by  $\si_t(b)$, where $1\leq t\leq n$ and $b$ ranges over all monomials in \begin{enumerate}
\item[$\bullet$]  $X_1,\ldots,X_d$ (see~\cite{Sibirskii_1968}, \cite{Procesi76}, \cite{Donkin92a}),

\item[$\bullet$]  $X_1,\ldots,X_d,X_1^T,\ldots,X_d^{T}$ (see~\cite{Procesi76},~\cite{Zubkov99}), where $p\neq2$, 

\item[$\bullet$] $Y_1,\ldots,Y_d$ (see~\cite{ZubkovI} or~\cite{Lopatin_so_inv}), where $p\neq2$, 

\item[$\bullet$] $Z_1,\ldots,Z_d$ (see~\cite{ZubkovI} or~\cite{Lopatin_so_inv}), where $p\neq2$, 
\end{enumerate}
respectively.  Note that in case $p=0$ or $p>n$ the algebras of invariants considered above are generated by $\tr(b)$, where $b$ is the same as above. 
In what follows, whenever we consider the orthogonal group $O(n)$ or algebras $R^{O(n)}$, $R_{+}^{O(n)}$, $R_{-}^{O(n)}$, we assume that $p\neq2$. The ideal of relations between the generators of $R^{GL(n)}$ was described in~\cite{Razmyslov74, Procesi76, Zubkov96}. In case $p=0$ relations  between generators of  $R^{O(n)}$  were computed in~\cite{Procesi76} and in case $p\neq2$ relations between generators of matrix $O(n)$-invariants were obtained in~\cite{Lopatin_Orel} and~\cite{Lopatin_Ofree}.

The elements of $R$, $R_{+}$, $R_{-}$, respectively, can be interpreted as polynomial functions from 
 \begin{enumerate}
 \item[$\bullet$] $H=M(n)\oplus \cdots\oplus M(n)$, 
 
 \item[$\bullet$] $H_{+}=S_{+}(n)\oplus \cdots\oplus S_{+}(n)$,
  
 \item[$\bullet$] $H_{-}=S_{-}(n)\oplus \cdots\oplus S_{-}(n)$,
 \end{enumerate}
 respectively, to $\FF$ as follows: $x_{ij}(k)$ sends $u=(A_1,\ldots,A_d)\in H$ to $(A_k)_{i,j}$. We can consider $H$ as $GL(n)$-module by the formula: $g\cdot v  = (g A_1 g^{-1}, \ldots, g A_d g^{-1})$ for $g\in GL(n)$ and $v=(A_1,\ldots,A_d)\in H$. Then $H_{+}$ and $H_{-}$ are $O(n)$-modules. 
 
 Assume that $(G,A,V)$ is one of the following triples: $(GL(n),R,H)$, $(O(n),R,H)$, $(O(n),R_{+},H_{+})$, $(O(n),R_{-},H_{-})$. Then
 $$A^G=\{f\in A\,|\,f(g\cdot v)=f(v) \text{ for all }g\in G,\; v\in V\}$$
 (see the papers above mentioned, where the generators for the algebras of invariants were found).
 
 The notion of separating invariants was  introduced in 2002 by Derksen and Kemper~\cite{DerksenKemper_book} as a weaker concept than generating invariants.   Given a subset $S$ of $A^G$, we say that elements $u,v$ of $V$ {\it are separated by $S$} if  exists an invariant $f\in S$ with $f(u)\neq f(v)$. If  $u,v\in V$ are separated by $A^G$, then we simply say that they {\it are separated}. A subset $S\subset A^G$ of the invariant ring is called {\it separating} if for any $u, v$ from $V$ that are separated we have that they are separated
 by $S$. Separating sets over finite fields were studied in a recent paper~\cite{Kemper_Lopatin_Reimers_2}. 
 
 It follows from more general result of Domokos~\cite{Domokos_2007} and Draisma, Kemper, Wehlau~\cite{DraismaKemperWehlau_2008} that for any $n>1$ there exists $C(n)$, which does not depend on $d$, such that the set of all elements of $A^G$ of degree less than $C(n)$ is separating for all $d$. On the other hand, in case $0<p\leq n$ a similar statement is not valid for generating systems for $R^{GL(n)}$  (see~\cite{DKZ_2002}) and $R^{O(n)}$  (see~\cite{Lopatin_IndOrt}).
 
In~\cite{Lopatin_separating2x2} it was established that the set  
$$\begin{array}{cl}
\tr(X_i), \det(X_i),& 1\leq i\leq d,\\
\tr(X_i X_j), & 1\leq i<j\leq d,\\
\tr(X_i X_j X_k),&  1\leq i<j<k\leq d.\\
\end{array}
$$
is a minimal (by inclusion) separating set for the algebra of matrix invariants $R^{GL(2)}$ for any $d\geq1$.  The case of three nilpotent $3\times 3$ matrices over an algebraically closed field of zero characteristic was considered in~\cite{Cavalcante_Lopatin_1}.  A minimal separating set for the algebra $R^{SL(2)\times SL(2)}$ of semi-invariants of $2\times 2$ matrices over an arbitrary algebraically closed field was explicitly described in~\cite{Domokos20, Domokos20Add}. 

In this paper we will establish minimal (by inclusion) separating sets for the following algebras of invariants:
\begin{enumerate}
 \item[$\bullet$]  $R^{GL(3)}$ for $d=2$ (see~Theorem~\ref{theo1xx}), namely, 
 $$\begin{array}{c}
\si_t(X_i),\; i=1,2,\;t=1,2,3;\\
\tr(X_1 X_2),\;  \tr(X_1^2 X_2),\;\tr(X_1 X_2^2),\; \tr(X_1^2 X_2^2),\;
\tr(X_1^2 X_2^2 X_1 X_2);\\
\end{array} 
$$

\item[$\bullet$] $R^{O(3)}_{-}$ for all $d>0$, where $p\neq 2$ (see Theorem~\ref{theo3xx}), namely,
$$\begin{array}{c}
\si_2(Z_i);\;\tr(Z_i Z_j),\; i<j;\; \tr(Z_i Z_j Z_k),\;  i<j<k,\\
\end{array}
$$
where $1\leq i,j,k\leq d$;

\item[$\bullet$] $R^{O(4)}_{-}$ for $d=2$, where $p\neq 2$  (see Theorem~\ref{theo4xx}), namely,
$$\begin{array}{c}
\si_2(Z_i),\; \det(Z_i),\;i=1,2;\\
\tr(Z_1 Z_2),\;\;  \si_2(Z_1 Z_2),\; \tr(Z_1^2 Z_2^2),\; \tr(Z_1^3 Z_2),\; \tr(Z_1 Z_2^3);\\
\end{array}
$$

\item[$\bullet$] $R_{+}^{O(3)}$ for $d=2$, where $p\neq 2$  (see Theorem~\ref{theo-O3-symm2}), namely,
$$\begin{array}{c}
\si_t(Y_i),\; i=1,2,\;t=1,2,3;\;\; \tr(Y_1 Y_2),\;  \tr(Y_1^2 Y_2),\;\tr(Y_1 Y_2^2),\; \tr(Y_1^2 Y_2^2).\\
\end{array}
$$
We will establish that the last set is a minimal generating set for $R_{+}^{O(3)}$ for $d=2$.
\end{enumerate}

We will also construct a minimal generating set for  $R_{+}^{O(3)}$ for $d=3$  (see Theorem~\ref{theo-O3-symm3}), namely,
$$\begin{array}{c}
\si_t(Y_i),\; i,t\in\{1,2,3\};\;\;
\tr(Y_i Y_j),\;  \tr(Y_i^2 Y_j),\;\tr(Y_i Y_j^2),\; \tr(Y_i^2 Y_j^2),\; 1\leq i<j\leq 3;\\
\tr(Y_1 Y_2 Y_3),\\
\tr(Y_1^2 Y_2 Y_3),\;\tr(Y_2^2 Y_1 Y_3),\;\tr(Y_3^2 Y_1 Y_2),\;\;
\tr(Y_1^2 Y_2^2 Y_3),\;\tr(Y_1^2 Y_3^2 Y_2),\;\tr(Y_2^2 Y_3^2 Y_1).\\
\end{array}
$$

Note that over a field of real numbers a minimal generating set for $R_{+}^{O(3)}$ for each $d>0$ was constructed by Spencer and Rivlin in series of papers~\cite{Spencer1,Spencer2,Spencer3,Spencer4} (see also~\cite{Spencer_book}).

%Let us remark that Jing~\cite{Jing_2015} established that in case $\FF=\CC$ the algebra of invariants $R^{O(n)}$ separates the orbits of the action of $O(n)$ on $H$. Hence the same result is obviously holds for the pairs $(R_{+}^{O(n)}, H_{+})$ and $(R_{-}^{O(n)}, H_{-})$.
% !!!??? Chto s Jingom delat' ??

%=========================================================
\subsection{Notations}
\label{section_notations}

%If for $A,B\in M(n)$ there exists $g\in GL(n)$ such that $gAg^{-1}=B$, then we write $A\sim B$. 
For a monomial $c\in R$ denote by $\deg{c}$ its {\it degree} and by $\mdeg{c}$ its {\it multidegree}, i.e., $\mdeg{c}=(t_1,\ldots,t_d)$, where $t_k$ is the total degree of the monomial $c$ in $x_{ij}(k)$, $1\leq i,j\leq n$, and $\deg{c}=t_1+\cdots+t_d$. Obviously, the algebras $R^{GL(n)}$, $R^{O(n)}$, $R_{+}^{O(n)}$, $R_{-}^{O(n)}$ have $\NN$-grading by degrees and $\NN^d$-grading by multidegrees. 

Denote by $E_{ij}$ the matrix such that the $(i,j)^{\rm th}$ entry is equal to one and the rest of entries are zeros. 
%The diagonal matrix with elements $a_1,\ldots,a_n$ we denote by $\diag(a_1,\ldots,a_n)$. 
For short, we write $J_3$ for $E_{12} + E_{23}$.

%=========================================================
%=========================================================
\section{Decomposable invariants}\label{section_dec}
In this section we assume $n=3$ and $d>0$. We consider some fact about decomposable invariants that we are going to apply later.  

Assume that a triple $(G,A,V)$ is the same as in Section~\ref{section_intro}. We say that an $\NN$-homogeneous invariant $f\in A^G$ is {\it decomposable} and write $f\equiv0$ if $f$ is a polynomial in $\NN$-homogeneous invariants of $A^G$ of strictly lower degree. If $f$ is not decomposable, then we say that $f$ is {\it indecomposable} and write $f\not\equiv0$. In case $f-h\equiv0$ we write $f\equiv h$. Denote by $P_{\FF}$ the ring of polynomials in $Y_1,\ldots,Y_d$ without free terms and by $P$ the set of (non-empty) products of $Y_1,\ldots,Y_d$ and set  $P_{1}=P\sqcup \{I_3\}$.

Consider the surjective homomorphism $\Psi: R^{GL(n)}\to R_{+}^{O(n)}$ defined by $x_{ij}(k)\to x_{ij}(k)$ in case $i\geq j$ and by $x_{ij}(k)\to x_{ji}(k)$ otherwise. Note that the image of $\tr(X_{i_1}\cdots X_{i_k})$ with respect to $\Psi$ is $\tr(Y_{i_1}\cdots Y_{i_k})$.

\begin{lemma}\label{lemma_dec}
Assume $p\neq 2$, $x,y\in P$ and $q\in P_{\FF}$. Then the next formulas hold in $R_{+}^{O(3)}$:
\begin{enumerate}
\item[(a)] $\tr(x y x^2 q)\equiv -\tr(x^2 y x q)$;

\item[(b)] $\tr(Y_1^2 Y_2^i Y_1 Y_3^j)\equiv 0$ for $i,j=1,2$;

\item[(c)] $\tr(Y_1^2 Y_2^2 Y_1 Y_2)\equiv 0$;

\item[(d)] $\tr(Y_1^2 Y_2^2 Y_3^2)\equiv 0$ if $p\neq 3$;

\item[(e)] $\tr(y^2 x^2 y x q)\equiv -\tr(x^2 y^2 x y q)$; 

\item[(f)] $\tr(Y_1^2 Y_2^2 Y_1 Y_2 Y_3^i)\equiv 0$ for $i=1,2$;

%\item[(g)] $\tr(Y_1^2 Y_2^2 Y_1 Y_3^2)\equiv 0$.

\end{enumerate}
\end{lemma}
\begin{proof}
\noindent {\bf (a)} Applying the homomorphism $\Psi$ to formula~(20) of~\cite{Lopatin_Comm1} we obtain the required. 

For the sake of completeness, we show that part (a) can also be proven by straightforward calculations. Namely, part (a) follows from the next equality, which holds for every $3\times 3$ matrices $A,B,C$ over any commutative ring:
$$\begin{array}{c}
\tr(A^2BAC) + \tr(ABA^2C) = \\
\tr(A) \biggl(- \tr(A^2CB) +  \tr(ABAC) + \tr(BA^2C)  - \tr(AB) \tr(AC) - \tr(B) \tr(A^2C)  \biggr) + \\
+\si_2(A) \biggl(  \tr(ACB) - \tr(BAC) + \tr(B) \tr(AC) \biggl)\\ 
- \det(A)\tr(BC)  + \tr(A^3C) \tr(B)  + \tr(A^2C) \tr(AB)  + \tr(AC) \tr(A^2B) .\\
\end{array}$$

\medskip
\noindent {\bf (b) } Applying the equality $\tr(A^T)=\tr(A)$ that holds for any matrix $A$  and part~(a) of the lemma we obtain 
$$ \tr(Y_1^2 Y_2^i Y_1 Y_3^j) = \tr(Y_3^j Y_1 Y_2^i Y_1^2) = \tr(Y_1 Y_2^i Y_1^2 Y_3^j) \equiv -\tr(Y_1^2 Y_2^i Y_1 Y_3^j)$$
in $R_{+}^{O(3)}$. Since $p\neq2$, the proof of part~(b) is completed. 

\medskip
\noindent {\bf (c) } Making the substitution $Y_3\to Y_2$ in part~(b) of this lemma, where $i=2$ and $j=1$, we obtain the required.

\medskip
\noindent {\bf (d) } Since $p\neq 3$, applying the homomorphism $\Psi$ to part~7 of Lemma~18
from~\cite{Lopatin_Comm2} we obtain
$$\tr(Y_1^2 Y_2^2 Y_3^2)\equiv -\tr(Y_1^2 Y_3^2 Y_2^2) = -\tr(Y_1^2 Y_2^2 Y_3^2).$$
Since we also have that $p\neq2$, the proof of part~(d) is completed.

\medskip
\noindent {\bf (e) } Formula~(14) of~\cite{Lopatin_Comm1}, which is valid for any $p$,  together with Lemma~3 of~\cite{Lopatin_Comm1} imply that the analogue of part~(e) of this lemma is valid for $R^{GL(3)}$. The application of homomorphism $\Psi$ concludes the proof of part~(e).

\medskip
\noindent {\bf (f) } Applying two times part~(a) of the lemma to 
$f=\tr(Y_3^i Y_2 Y_1 Y_2^2 Y_1^2)$ we obtain that 
$$f\equiv\tr(Y_3^i Y_2^2 Y_1^2 Y_2 Y_1)\equiv -\tr(Y_3^i Y_1^2 Y_2^2 Y_1 Y_2),$$
where the second equivalence follows from part~(e) of the lemma. On the other hand, $f=\tr(Y_3^i Y_1^2 Y_2^2 Y_1 Y_2)$ and the proof of part~(f) is completed. 

%\medskip
%\noindent {\bf (g) } Applying part~(a) of the lemma, we have
%$$\tr(Y_1^2 Y_2^2 Y_1 Y_3^2) = \tr(Y_3^2 Y_1 Y_2^2 Y_1^2)\equiv 
%-\tr(Y_3^2 Y_1^2 Y_2^2 Y_1) = - \tr(Y_1^2 Y_2^2 Y_1 Y_3^2).$$
%The proof of part~(g) is completed. 
\end{proof}

%=========================================================
%=========================================================
\section{Invariants of two $3\times 3$ matrices}
\label{section_newresult_2GL3}

\begin{theo}\label{theo1xx}
For $d=2$ the following set is a minimal separating set for the algebra $R^{GL(3)}$ of  $GL(3)$-invariants of two matrices: 
$$\begin{array}{c}
\tr(X_i),\;\si_2(X_i),\;\det(X_i),\; i=1,2,\\
\tr(X_1 X_2),\;  \tr(X_1^2 X_2),\;\tr(X_1 X_2^2),\; \tr(X_1^2 X_2^2),\\
\tr(X_1^2 X_2^2 X_1 X_2) .\\
\end{array}
$$
\end{theo}
\begin{proof}
Denote the set from the formulation of the theorem by $S$. It is well-known that $S$ generates $R^{GL(3)}$ (for example, see~\cite{Lopatin_Sib}). Thus $S$ is a separating set. To prove that $S$ is a minimal separating set we will show that for any  element $f\in S$ the set $S_0=S\backslash \{f\}$ is not separating. 

The case of $f$ from the list $\tr(X_i)$, $\si_2(X_i)$, $\det(X_i)$ $(i=1,2)$ is obvious. 

Assume $f=\tr(X_1 X_2)$. Then for $A_1=B_1=J_3$ (see Section~\ref{section_notations}), $A_2=E_{32}$, $B_2=E_{12}$ we have that $(A_1,A_2)$ and  $(B_1, B_2)$ are not separated by $S_0$, but $f$ separates $(A_1,A_2)$ and $(B_1, B_2)$. Thus $S_0$ is not a separating set. 

Assume $f=\tr(X_1^2 X_2)$. Then for 
$$A_1=B_1=J_3,\qquad
A_2 = \left(
\begin{array}{ccc}
0 & 1 & 1\\ 
0 & 1 & -1\\ 
0 & 1 &-1 \\
\end{array}
\right),\quad 
B_2 = \left(
\begin{array}{ccc}
0 & 0 & 0\\ 
1 & 0 & 0\\ 
-1 & 0 &0 \\
\end{array}
\right)
$$
we have that $(A_1,A_2)$ and  $(B_1, B_2)$ are not separated by $S_0$, but $f$ separates $(A_1,A_2)$ and $(B_1, B_2)$. Thus $S_0$ is not a separating set. Similarly we consider the case of $f=\tr(X_1 X_2^2)$.

Assume $f=\tr(X_1^2 X_2^2)$. Then for
$$A_1=B_1=J_3,\qquad
A_2 = \left(
\begin{array}{ccc}
0 & 1 & 0\\ 
0 & -1 & 0\\ 
1 & 1 & 1 \\
\end{array}
\right),\quad 
B_2 = \left(
\begin{array}{ccc}
0 & 1 & 0\\ 
1 & 0 & 0\\ 
1 & 0 &0 \\
\end{array}
\right)
$$
we have that $(A_1,A_2)$ and  $(B_1, B_2)$ are not separated by $S_0$, but $f$ separates $(A_1,A_2)$ and $(B_1, B_2)$. Thus $S_0$ is not a separating set.

Assume $f=\tr(X_1^2 X_2^2 X_1 X_2)$. Then for
$$A_1=B_1=J_3,\qquad
A_2 = \left(
\begin{array}{ccc}
1 & 0 & 0\\ 
1 & -1 & 0\\ 
1 & -1 & 0 \\
\end{array}
\right),\quad 
B_2 = \left(
\begin{array}{ccc}
-1 & 0 & 0\\ 
0 & 0 & 1\\ 
1 & 0 &1 \\
\end{array}
\right)
$$
we have that $(A_1,A_2)$ and  $(B_1, B_2)$ are not separated by $S_0$, but $f$ separates $(A_1,A_2)$ and $(B_1, B_2)$. Thus $S_0$ is not a separating set. The theorem is proven.
\end{proof}

%=======================================================
%=======================================================
\section{Orthogonal invariants of several $3\times 3$ skew-symmetric  matrices}
\label{section_newresult_skewO3}

\begin{lemma}\label{lemma31xx}
The set  
$$
\si_2(Z_1),\; \si_2(Z_2),\; \tr(Z_1 Z_2)
$$
is a minimal separating set for the algebra $I=R_{-}^{O(3)}$ for $d=2$ in case $p\neq 2$.
\end{lemma}
\begin{proof}
Denote the set from the formulation of the lemma by $S$. The set $S$ generates the algebra of invariants $I$ (for example, see Theorem 1.4 of~\cite{Lopatin_Oskew}). Thus $S$ is a separating set. To prove that $S$ is a minimal separating set we will show that for any  element $f\in S$ the set $S_0=S\backslash \{f\}$ is not separating. 

The case of $f=\si_2(Z_i)$ ($i=1,2$) is obvious.

Assume $f=\tr(Z_1 Z_2)$. Then for $A_1=A_2=B_1=-B_2=E_{12}-E_{21}$ we have that $(A_1,A_2)$ and $(B_1,B_2)$ are not separated by $S_0$, but $f$ separates $(A_1,A_2)$ and $(B_1,B_2)$. Thus $S_0$ is not a separating set. The lemma is proven.
\end{proof}

\begin{theo}\label{theo3xx}
Assume $d>0$ and  $p\neq 2$. Then the set  
$$\begin{array}{c}
\si_2(Z_i),\;1\leq i\leq d,\\
\tr(Z_i Z_j),\;  1\leq i<j\leq d,\\
\tr(Z_i Z_j Z_k),\;  1\leq i<j<k\leq d,\\
\end{array}
$$
is a minimal separating set for the algebra $R_{-}^{O(3)}$ of $O(3)$-invariants of $d$ skew-symmetric matrices.
\end{theo}
\begin{proof} Denote the set from the formulation of the theorem by $S$. By Theorem 1.4 of~\cite{Lopatin_Oskew} the set $S$ generates the algebra of invariants $R^{O(3)}_{-}$. Hence $S$ is a separating set for $R_{-}^{O(3)}$. Applying Lemma~\ref{lemma31xx} we obtain that to prove that $S$ is a minimal separating set it is enough to show that in case $d=3$ the set $S_0=S\backslash \{\tr(Z_1Z_2Z_3)\}$ is not separating. 

For $A_1=B_1=E_{12}-E_{21}$, $A_3=B_3=E_{13}-E_{31}$, and 
$$A_2 = \left(
\begin{array}{ccc}
0 & 0 & 1\\ 
0 & 0 & 1\\ 
-1 & -1 & 0 \\
\end{array}
\right),\quad 
B_2 = \left(
\begin{array}{ccc}
0 & 0 & 1\\ 
0 & 0 & -1\\ 
-1 & 1 &0 \\
\end{array}
\right)
$$
we have that $u=(A_1,A_2,A_3)$ and  $v=(B_1,B_2,B_3)$ are not separated by $S_0$, but $\tr(Z_1 Z_2 Z_3)$ separates $u$ and $v$ in case $p\neq2$. Thus $S_0$ is not a separating set.
The theorem is proven.
\end{proof}

%=======================================================
%=======================================================
\section{Orthogonal invariants of two $4\times 4$ skew-symmetric  matrices}
\label{section_newresult_2skewO4}

\begin{theo}\label{theo4xx}
Assume $d=2$ and $p\neq 2$. Then the set  
$$\begin{array}{c}
\si_2(Z_i),\; \det(Z_i),\;i=1,2,\\
\tr(Z_1 Z_2),\;  \si_2(Z_1 Z_2),\; \tr(Z_1^2 Z_2^2),\; \tr(Z_1^3 Z_2),\;\tr(Z_1 Z_2^3),\\
\end{array}
$$
is a minimal separating set for the algebra $R_{-}^{O(4)}$ of $O(4)$-invariants of two skew-symmetric matrices.
\end{theo}
\begin{proof} Assume $d=2$. Denote the set from the formulation of the theorem by $S$. By Theorem~1.1 of~\cite{Lopatin_mgs_skewO4},  the set $S$ generates the algebra of invariants $R^{O(4)}_{-}$. Hence $S$ is a separating set for $R_{-}^{O(4)}$. To prove that $S$ is a minimal separating set we will show that for any  element $f\in S$ the set $S_0=S\backslash \{f\}$ is not separating.

If $f$ is $\si_2(Z_1)$ or $\tr(Z_1 Z_2)$, then $S_0$ is not separating by Lemma~\ref{lemma31xx}. 

Assume $f=\det(Z_1)$. For 
$$A_1 = \left(
\begin{array}{cccc}
 0 & 0 & 0 & 1\\ 
 0 & 0 & 1 & 0\\
 0 &-1 & 0 & 0\\ 
-1 & 0 & 0 & 0\\ 
\end{array}
\right),\qquad
B_1 = \left(
\begin{array}{cccc}
 0 & 1 & 0 & 1\\ 
-1 & 0 & 0 & 0\\
 0 & 0 & 0 & 0\\ 
-1 & 0 & 0 & 0\\ 
\end{array}
\right),
$$
$A_2 = B_2 = 0$ we have that $u=(A_1,A_2)$ and  $v=(B_1,B_2)$ are not separated by $S_0$, but $\det(Z_1)$ separates $u$ and $v$. Thus $S_0$ is not a separating set.

Assume $f=\si_2(Z_1 Z_2)$. For 
$$A_1 = \left(
\begin{array}{cccc}
 0 & 1 & 0 & 0\\ 
-1 & 0 & 0 & 0\\
 0 & 0 & 0 &-1\\ 
 0 & 0 & 1 & 0\\ 
\end{array}
\right),\quad
A_2 = B_2 = \left(
\begin{array}{cccc}
 0 & 0 & 1 & 0\\ 
 0 & 0 & 0 & 1\\
-1 & 0 & 0 & 0\\ 
 0 &-1 & 0 & 0\\ 
\end{array}
\right),
\quad
B_1 = \left(
\begin{array}{cccc}
 0 &-1 & 0 & 0\\ 
 1 & 0 & 0 & 0\\
 0 & 0 & 0 &-1\\ 
 0 & 0 & 1 & 0\\ 
\end{array}
\right)
$$
we have that $u=(A_1,A_2)$ and  $v=(B_1,B_2)$ are not separated by $S_0$, but $\si_2(Z_1 Z_2)$ separates $u$ and $v$. Thus $S_0$ is not a separating set.

Assume $f=\tr(Z_1^2 Z_2^2)$. For 
$$A_1 = \left(
\begin{array}{cccc}
 0 & 1 & 1 & 1\\ 
-1 & 0 &-1 &-2\\
-1 & 1 & 0 &-1\\ 
-1 & 2 & 1 & 0\\ 
\end{array}
\right),\qquad
A_2 = \left(
\begin{array}{cccc}
 0 &-1 & 0 & 1\\ 
 1 & 0 & 0 & 0\\
 0 & 0 & 0 & 0\\ 
-1 & 0 & 0 & 0\\ 
\end{array}
\right),$$
$$B_1 = \left(
\begin{array}{cccc}
 0 &-2 & 0 &-1\\ 
 2 & 0 & 0 & 2\\
 0 & 0 & 0 & 0\\ 
 1 &-2 & 0 & 0\\ 
\end{array}
\right),
\qquad
B_2 = \left(
\begin{array}{cccc}
 0 & 0 & 1 & 0\\ 
 0 & 0 & 0 & 0\\
-1 & 0 & 0 & 1\\ 
 0 & 0 &-1 & 0\\ 
\end{array}
\right)
$$
we have that $u=(A_1,A_2)$ and  $v=(B_1,B_2)$ are not separated by $S_0$, but $\tr(Z_1^2 Z_2^2)$ separates $u$ and $v$. Thus $S_0$ is not a separating set.

Assume $f=\tr(Z_1 Z_2^3)$. For 
$$A_1 = B_1 = \left(
\begin{array}{cccc}
 0 & 0 & 1 & 1\\ 
 0 & 0 & 1 & 1\\
-1 &-1 & 0 & 0\\ 
-1 &-1 & 0 & 0\\ 
\end{array}
\right),\;
A_2 = \left(
\begin{array}{cccc}
 0 & 1 & 0 & 1\\ 
-1 & 0 & 0 &-1\\
 0 & 0 & 0 & 1\\ 
-1 & 1 &-1 & 0\\ 
\end{array}
\right),\;
B_2 = \left(
\begin{array}{cccc}
 0 & 1 & 0 &-1\\ 
-1 & 0 & 0 & 1\\
 0 & 0 & 0 & 1\\ 
 1 &-1 &-1 & 0\\ 
\end{array}
\right)
$$
we have that $u=(A_1,A_2)$ and  $v=(B_1,B_2)$ are not separated by $S_0$, but $\tr(Z_1 Z_2^3)$ separates $u$ and $v$. Thus $S_0$ is not a separating set. The case of $f=\tr(Z_1^3 Z_2)$ is similar. The theorem is proven.
\end{proof}

%=========================================================
%=========================================================
\section{Orthogonal invariants of two $3\times 3$ symmetric matrices}
 
\begin{theo}\label{theo-O3-symm2} Assume that $p\neq 2$ and $d=2$. Then the set
$$\begin{array}{c}
\tr(Y_i),\;\si_2(Y_i),\;\det(Y_i),\; i=1,2,\\
\tr(Y_1 Y_2),\;  \tr(Y_1^2 Y_2),\;\tr(Y_1 Y_2^2),\; \tr(Y_1^2 Y_2^2)\\
\end{array}
$$
is a minimal generating set and a minimal separating set for the algebra of $O(3)$-invariants $R_{+}^{O(3)}$ of two symmetric matrices.
\end{theo} 
\begin{proof}
Denote by $S$ the set from the formulation of the theorem and by $S_{X}$ the result of substitutions $Y_i\to X_i$ ($i=1,2$) in $S$. It is well-known that $S_X\cup\{\tr(X_1^2 X_2^2 X_1 X_2)\}$ generates $R^{GL(3)}$ in case $d=2$ (for example, see~\cite{Lopatin_Sib}). Considering the  surjective homomorphism  $\Psi$ from Section~\ref{section_dec}  we obtain that $S\cup f$  generates $R_{+}^{O(3)}$, where $f=\tr(Y_1^2 Y_2^2 Y_1 Y_2)$. Since $f$ is decomposable in $R_{+}^{O(3)}$ by part~(b) of Lemma~\ref{lemma_dec}, then $S$ generates $R_{+}^{O(3)}$. Thus $S$ is a separating set. To prove that $S$ is a minimal separating set we will show that for any  element $f\in S$ the set $S_0=S\backslash \{f\}$ is not separating. 

Assume $f=\tr(Y_1 Y_2)$. Then for 
$$A_1=B_1=\symmmatr{0}{0}{1}{1}{0}{0}, \quad A_2=  \symmmatr{0}{0}{0}{0}{1}{0}, \quad B_2=\symmmatr{1}{0}{0}{-1}{0}{0}$$
we have that $(A_1,A_2)$ and  $(B_1, B_2)$ are not separated by $S_0$, but $f$ separates $(A_1,A_2)$ and $(B_1, B_2)$. Thus $S_0$ is not a separating set. 

Assume $f=\tr(Y_1^2 Y_2)$. Then for 
$$A_1=B_1=\symmmatr{0}{0}{0}{0}{1}{0}, \quad A_2=  \symmmatr{0}{1}{0}{0}{0}{0}, \quad B_2=\symmmatr{1}{0}{0}{-1}{0}{0}$$
we have that $(A_1,A_2)$ and  $(B_1, B_2)$ are not separated by $S_0$, but $f$ separates $(A_1,A_2)$ and $(B_1, B_2)$. Thus $S_0$ is not a separating set. 
The case of $f=\tr(Y_1 Y_2^2)$ is similar. 

Assume $f=\tr(Y_1^2 Y_2^2)$. Then for 
$$A_1=\symmmatr{0}{1}{0}{0}{0}{0}, \quad A_2=B_1=  \symmmatr{0}{0}{0}{0}{1}{0}, \quad B_2=\symmmatr{0}{0}{0}{1}{0}{-1}$$
we have that $(A_1,A_2)$ and  $(B_1, B_2)$ are not separated by $S_0$, but $f$ separates $(A_1,A_2)$ and $(B_1, B_2)$. Thus $S_0$ is not a separating set. 

The cases of $f=\si_k(Y_i)$ ($i=1,2$, $k=1,2,3$) are trivial. For the sake of completeness, we point out that in case $f=\det(Y_1)$ we consider
$$A_1=-B_1=\symmmatr{1}{0}{0}{1}{0}{-2},\quad A_2=B_2=0.$$

Therefore,  $S$ is a minimal separating set. This result together with the fact that  $S$ is a generating set imlply that $S$ is a minimal generating set for $R_{+}^{O(3)}$. 
\end{proof}

%=========================================================
%=========================================================
\section{Orthogonal invariants of three $3\times 3$ symmetric matrices}

\begin{theo}\label{theo-O3-symm3}
For $d=3$ consider the following set $S$:
$$\begin{array}{c}
\tr(Y_i),\;\si_2(Y_i),\;\det(Y_i),\; i=1,2,3,\\
\tr(Y_i Y_j),\;  \tr(Y_i^2 Y_j),\;\tr(Y_i Y_j^2),\; \tr(Y_i^2 Y_j^2),\; 1\leq i<j\leq 3,\\
\tr(Y_1 Y_2 Y_3),\\
\tr(Y_1^2 Y_2 Y_3),\;\tr(Y_2^2 Y_1 Y_3),\;\tr(Y_3^2 Y_1 Y_2),\\
\tr(Y_1^2 Y_2^2 Y_3),\;\tr(Y_1^2 Y_3^2 Y_2),\;\tr(Y_2^2 Y_3^2 Y_1).\\
\end{array}
$$
Then the set
\begin{enumerate}
\item[$\bullet$] $S$, if $p\neq 2,3$,

\item[$\bullet$] $S\cup \{\tr(Y_1^2 Y_2^2 Y_3^2)\}$, if $p=3$,
\end{enumerate}
is a minimal generating set for the algebra of $O(3)$-invariants $R_{+}^{O(3)}$ of three symmetric matrices.
\end{theo}
\begin{proof}Denote by $S_1$ the set from the formulation of the theorem. We split the proof into two parts, namely, at first we show that $S_1$ generates $R_{+}^{O(3)}$ and then we prove that $S_1$ is minimal.

\smallskip
\noindent{\bf (a)} We apply $\Psi$ to the (minimal) generating set for $R^{GL(3)}$ from Theorem~1 of~\cite{Lopatin_Sib} and obtain that $R_{+}^{O(3)}$ is generated by $S\cup G_1$ in case $p\neq 2,3$ and by $S\cup G_1\cup G_2$ in case $p=3$. Here $G_1$ is the set
$$\begin{array}{c}
f_{ij}=\tr(Y_i^2 Y_j^2 Y_i Y_j),\; i<j;\;\;\; \tr(Y_1 Y_3 Y_2); \;\;\; h_1=\tr(Y_1^2 Y_2^2 Y_3^2); \\
\tr(Y_i^2 Y_j Y_k),\; j>k;\;\;\; \tr(Y_i^2 Y_j^2 Y_k),\; i>j; \\
r_{ijk}=\tr(Y_i^2 Y_j Y_i Y_k),\; j<k;\;\;\; s_{ijk}=\tr(Y_i^2 Y_j^2 Y_i Y_k),\; 
\end{array}
$$
where $1\leq i,j,k\leq 3$ are pairwise different, and $G_2$ is the set
$$\begin{array}{c}
h_2=\tr(Y_1^2 Y_3^2 Y_2^2),\\
a_{ijk}=\tr(Y_i Y_j^2 Y_k^2 Y_j Y_k),\;\;\; b_{ijk}=\tr(Y_i^2 Y_j^2 Y_i Y_k^2),\;\;\; c_{ijk}=\tr(Y_i^2 Y_j^2 Y_k^2 Y_j Y_k),\\
\end{array}
$$
where $j<k$, $1\leq i,j,k\leq 3$ are pairwise different. Parts~(b), (c), (f) of Lemma~\ref{lemma_dec} imply that $f_{ij}$, $r_{ijk}$, $s_{ijk}$, $a_{ijk}$, $b_{ijk}$, $c_{ijk}$ are decomposable in $R_{+}^{O(3)}$. It follows from part~(d) of  Lemma~\ref{lemma_dec} that $h_1$ and $h_2$ are decomposable in $R_{+}^{O(3)}$ in case $p\neq3$. Finally, the equalities $\tr(Y_1 Y_3 Y_2) = \tr(Y_1 Y_2 Y_3)\in S$,  $\tr(Y_i^2 Y_j Y_k) = \tr(Y_i^2 Y_k Y_j)$, $\tr(Y_i^2 Y_j^2 Y_k) = \tr(Y_j^2 Y_i^2 Y_k)$,  and $h_1=h_2$ imply that $S_1$ generates the algebra $R_{+}^{O(3)}$.

\medskip
\noindent{\bf (b)} Since all elements of $S$ have pairwise different multidegrees, to show that $S$ is a minimal generating set it is enough to establish that all elements of $S_1$ are indecomposable in $R_{+}^{O(3)}$. Then by Theorem~\ref{theo-O3-symm2} we can only verify the following claims:
\begin{enumerate}
\item[(1)] $f_1=\tr(Y_1 Y_2 Y_3)$ is indecomposable,

\item[(2)] $f_2=\tr(Y_1^2 Y_2 Y_3)$  is indecomposable,

\item[(3)] $f_3=\tr(Y_1^2 Y_2^2 Y_3)$ is indecomposable,

\item[(4)] $f_4=\tr(Y_1^2 Y_2^2 Y_3^3)$ is indecomposable in case $p=3$.
\end{enumerate}

Assume that $f_1$ is decomposable, i.e., $f_1$ is a polynomial in invariants of lower degree. Then it is easy to see that
\begin{eq}\label{eq1}
\tr(A_1 A_2 A_3)=0
\end{eq}%
for all nilpotent  matrices $A_1, A_2, A_3\in S_{+}(3)$ with entries in  $\FF$. Moreover, equality~(\ref{eq1}) is valid for any extension of $\FF$. In particular, we can assume that $\FF$ is algebraically closed. 

We will use the following symmetric nilpotent matrices as tests:
$$R_1 = \symmmatr{0}{1}{0}{0}{\I}{0},\; R_2 = \symmmatr{1}{1}{0}{-3}{2\I \sqrt{2}}{2},\; R_3 = \symmmatr{0}{-1}{0}{0}{\I}{0}$$
$$T_1 = \symmmatr{1}{\I}{0}{-1}{0}{0},\; 
T_2 =  \symmmatr{1}{-\I}{0}{-1}{0}{0},\;
T_3 =  \symmmatr{1}{0}{\I}{0}{0}{-1},\; 
$$
where $\I^2=-1$. Note that $T_i^2=0$ for $i=1,2,3$. 

We have $\tr(T_1 T_2 T_3)=2\neq0$; a contradiction to equality~(\ref{eq1}).

Assume that $f_2$ is decomposable. Then it is easy to verify that there exists $\al\in\FF$ such that
\begin{eq}\label{eq2}
\tr(A_1^2 A_2 A_3)= \al \tr(A_1 A_2) \tr(A_1 A_3)
\end{eq}%
for all nilpotent  matrices $A_1, A_2, A_3\in S_{+}(3)$ with entries in  $\FF$, which we assume to be algebraically closed. For $\un{A}=(A_1,A_2,A_3)=(T_1,T_2,T_3)$ equality~(\ref{eq2}) implies $\al=0$. Hence for $\un{A}=(R_1,R_2,T_1)$ equality~(\ref{eq2}) imply $1+(1-2\sqrt{2})\I=0$. It is easy to see that the obtained equality does not hold in case $p\neq2$.

Assume that $f_3$ is decomposable. Then it is easy to verify that there exist $\al,\be,\ga\in\FF$ such that
\begin{eq}\label{eq3}
\tr(A_1^2 A_2^2 A_3)= \al \tr(A_1 A_2) \tr(A_1 A_2 A_3) + 
\be \tr(A_1 A_3) \tr(A_1 A_2^2) +
\ga \tr(A_2 A_3) \tr(A_1^2 A_2)
\end{eq}%
for all nilpotent matrices $A_1, A_2, A_3\in S_{+}(3)$ with entries in  $\FF$, which we assume to be algebraically closed. Considering  $\un{A}=(T_1, T_2,T_3)$ in equality~(\ref{eq3}) we obtain $\al=0$. For $\un{A}=(T_1, R_1, T_2)$  equality~(\ref{eq3}) implies $\be=0$. If $\un{A}=(R_1, T_1, T_2)$ in equality~(\ref{eq3}), then we have $\ga=0$. Finally, for $\un{A}=(R_1, R_2, T_1)$ equality~(\ref{eq3}) imply $2(\I-1)(\sqrt{2}-1)=0$; a contradiction. 

Assume that $f_4$ is decomposable. Then it is easy to verify that there exist $\al_i,\be_i,\ga,\de\in\FF$ ($i=1,2,3$) such that  $\tr(A_1^2 A_2^2 A_3^3)=$ 
$$\begin{array}{c}
\al_1  \tr(A_1^2 A_2 A_3) \tr(A_2 A_3) + \al_2  \tr(A_2^2 A_1 A_3) \tr(A_1 A_3) +
   \al_3  \tr(A_3^2 A_1 A_2) \tr(A_1 A_2) + \\
  \be_1 \tr(A_1^2 A_3) \tr(A_2^2 A_3) + \be_2 \tr(A_1^2 A_2) \tr(A_3^2 A_2) +
  \be_3 \tr(A_2^2 A_1) \tr(A_3^2 A_1) +\\   
  \ga \tr(A_1 A_2 A_3)^2  + \de \tr(A_1 A_2) \tr(A_1 A_3) \tr(A_2 A_3)\\
\end{array}$$%
for all nilpotent matrices $A_1, A_2, A_3\in S_{+}(3)$ with entries in  $\FF$, which we assume to be algebraically closed. Making the substitution $\un{A}=(T_1,T_2,T_3)$ in the above equality we obtain $\de=-\ga$. Then,  consequently considering substitutions $\un{A}=(R_1, T_1, T_2)$, $\un{A}=(T_1,R_1, T_2)$ and $\un{A}=(T_1, T_2, R_1)$ we obtain that $\al_1=\al_2=\al_3=2\ga$. Similarly, substitutions $\un{A}=(R_1,R_2,T_1)$, $\un{A}=(R_1,T_1,R_2)$ and $\un{A}=(T_1,R_1,R_2)$ imply that $\be_1=\be_2=\be_3=-\ga$. Finally, applying the substitution $\un{A}=(R_1,R_2,R_3)$ we get $6 \ga = -1$, which is a contradiction in case $p=3$. Therefore, $S_1$ is a minimal generating set for $R^{O(3)}_{+}$.
\end{proof}

\end{document}